\documentclass[12 pt]{amsart}
\usepackage[left=8em, right=8em, top=7em]{geometry}
\usepackage{comment}
\usepackage{amsmath}
\usepackage{mathtools}
\usepackage{color}
\usepackage{amsfonts}
\usepackage{enumerate}
\usepackage{amsthm}
\usepackage{chngcntr}
\usepackage{enumitem}
\usepackage[numbers]{natbib}
\usepackage{array}
\usepackage{titlesec}
\usepackage{hyperref}
\usepackage{comment}

\titleformat{\section}
  {\normalfont\fontsize{15}{15}\bfseries}{\thesection}{1em}{}

\newtheorem{theorem}{Theorem}[section]
\newtheorem{corollary}[theorem]{Corollary}
\newtheorem{lemma}[theorem]{Lemma}
\newtheorem{prop}[theorem]{Proposition}
\newtheorem{remark}[theorem]{Remark}

\newtheorem{defn}[theorem]{Definition}
\newtheorem{ex}[theorem]{Example}

\DeclareMathOperator{\ric}{Ric}
\DeclareMathOperator{\hess}{Hess}
\DeclareMathOperator{\Lie}{\mathcal{L}}
 
\DeclareMathOperator{\vol}{vol}
\DeclareMathOperator{\divergence}{div}

\DeclareMathOperator{\tr}{tr}
\DeclareMathOperator{\proj}{proj}

\newcommand{\bigslant}[2]{{\raisebox{.2em}{$#1$}\left/\raisebox{-.2em}{$#2$}\right.}}

\allowdisplaybreaks

\makeatletter
\renewenvironment{proof}[1][\proofname]{\par
  \pushQED{\qed}%
  \normalfont \topsep6\p@\@plus6\p@\relax
  \trivlist
  \item[\hskip\labelsep
        \itshape
    #1\@addpunct{.}]\mbox{}\\*
}{%
 \popQED\endtrivlist\@endpefalse
}
\makeatother

\begin{document}
\title{\vspace{-2cm} Locally Homogeneous Non-gradient Quasi-Einstein 3-Manifolds}
\author{Alice Lim}
\address{215 Carnegie Building\\
Dept. of Math, Syracuse University\\
Syracuse, NY, 13244.}
\email{awlim100@syr.edu}
\urladdr{https://awlim100.expressions.syr.edu}

\maketitle

\begin{abstract} In this paper, we classify the compact locally homogeneous non-gradient $m$-quasi Einstein 3-manifolds. Along the way, we also prove that given a compact quotient of a Lie group of any dimension that is $m$-quasi Einstein, the potential vector field $X$ must be left invariant and Killing. We also classify the nontrivial $m$-quasi Einstein metrics that are a compact quotient of be the product of two Einstein metrics. We also show that $S^1$ is the only compact manifold of any dimension which admits a metric which is nontrivially $m$-quasi Einstein and Einstein. 
\end{abstract}

\section{Introduction}

Non-gradient $m$-quasi Einstein manifolds are of particular interest in the study of near-horizon geometries (See \cite{KhuriWoolgar}, \cite{KhWoWy}, and \cite{ExtremalBlackHoles}). In this paper, we study non-gradient $m$-quasi Einstein manifolds as a generalization of Einstein manifolds, gradient $m$-quasi Einstein manifolds, and Ricci solitons. In order to define the $m$-quasi Einstein equation, we must first give the definition of the $m$-Bakry \'Emery Ricci tensor:
\vspace{1em}

\begin{defn}\label{defn:bakryemeryricci}
Let $X$ be a vector field on a Riemannian manifold $(M^n,g)$. The $m$-Bakry-\' Emery tensor is $$\ric_X^m:=\ric+\frac{1}{2}\Lie_X g-\frac{1}{m}X^*\otimes X^*$$

where $\Lie_X g$ is the Lie derivative of $g$ with respect to $X$, and \begin{align}
X^*  : T_p M & \rightarrow \mathbb{R} \nonumber \\
 Y & \mapsto g(X,Y)\nonumber.
\end{align}
\end{defn}

If $X=\nabla \phi$ where $\phi:M\to\mathbb{R}$ is a smooth function, the $m$-Bakry-\' Emery Ricci tensor is  $$\ric_\phi^m:=\ric+\hess\phi-\frac{1}{m}d\phi\otimes d\phi,$$ and we call this the gradient $m$-Bakry \'Emery Ricci tensor. Notice that when $\phi$ is a constant, the gradient $m$-Bakry \'Emery Ricci tensor is the Ricci tensor. If $m=\infty$, the $m$-Bakry-\'Emery Ricci tensor becomes $\ric+\frac{1}{2}\mathcal{L}_Xg$.
\vspace{1em}

The $\infty$-Bakry \'Emery Ricci curvature was first studied by Lichnerowicz in 1971 in \cite{Lich}, and Qian first studied the gradient $m$-Bakry \'Emery Ricci curvature with $m\neq \infty$ in \cite{ZhongminQian}. Bakry and \'Emery further studied the Bakry \'Emery Ricci curvature in relation to diffusion processes in \cite{BE}. They also arise in the study of optimal transport, Ricci flow, and general relativity. In \cite{Lott}, Lott gives topological consequences and relations to the measured Gromov-Hausdorff limits to lower bounds on the Bakry-\'Emery Ricci curvature. Wei-Wylie prove Bakry-\'Emery Ricci curvature analogs of the comparison theorems and the volume comparison theorem in \cite{WeiWylie}. There have been many more papers written about the subject, too many to summarize here. Now, we are ready to define the $m$-quasi Einstein equation.

\begin{defn}
A manifold $(M,g)$ satisfies the $m$-quasi Einstein equation if $\ric_X^m=Ag$ for some constants $A$. 
\end{defn}

\begin{remark}
Many authors only consider the gradient case and/or the manifolds with boundary case of the $m$-quasi Einstein equation. We will assume neither condition in this paper.
\end{remark}

The $m=\infty$ case of the $m$-quasi Einstein equation corresponds to the Ricci soliton equation, $\ric+\frac{1}{2}\mathcal{L}_Xg=Ag$. Ivey showed in \cite{Ivey} that compact Ricci solitons must be shrinking, i.e. $A$ must be positive. Perelman showed in \cite{Perelman} that compact shrinking Ricci solitons must be gradient. Then Petersen-Wylie showed in \cite{Petersen-Wylie} that any compact locally homogeneous gradient Ricci soliton is Einstein. Therefore, by Ivey, Perelman, and Petersen-Wylie, here are no non-Einstein non-trivial locally homogeneous compact Ricci solitons. 
\vspace{1em}

If $(M,g)$ is $m$-quasi Einstein and if $X=\nabla \phi$, then we call the space gradient $m$-quasi Einstein. If $X=0$, then we call the space trivial. Our first result is the following theorem and gives us a classification of manifolds which are Einstein and $m$-quasi Einstein.
\vspace{1em}

\begin{theorem}
Let $M^n$ be a compact Einstein manifold. Then $M$ is non-trivial $m$-quasi Einstein for $m\neq \infty$ if and only if $M$ is $S^1$. 
\end{theorem}

Gradient $m$-quasi Einstein metrics with $m>0$ where first systematically considered by Case-Shu-Wei in \cite{CaseShuWei} and Kim-Kim in \cite{Kim-Kim}. They show that gradient $m$-quasi Einstein metrics correspond to warped product Einstein metrics. In \cite[Theorem~2.1]{CaseShuWei}, Case-Shu-Wei prove that a compact gradient $m$-quasi Einstein with constant curvature must be trivial if $m>0$. Since locally homogeneous manifolds have constant scalar curvature, this shows that compact locally homogeneous manifolds which satisfy $\ric_\phi^m=Ag$ with $m>0$ must be trivial. The $m<0$ case follows from \cite[Theorem~1.9]{Petersen-Wylie-arxiv}. In \cite[Theorem~1.3]{HePetersenWylie_Constant_Scalar_Curvature}, He-Petersen-Wylie prove that if $(M^3,g)$ has no boundary, satisfies $\ric_\phi^m=Ag$ with $m>1$, and has constant scalar curvature, then $M^3$ is a quotient of $S^3$, $S^2\times \mathbb{R}$, $\mathbb{R}^3$, $H^2\times \mathbb{R}$, or $H^3$ with the standard metric. In \cite[Theorem~1.4]{HePetersenWylie_On_Homogeneous_Spaces_and_Homogeneous_Ricci_Solitons}, He-Petersen-Wylie show that if $(M^n,g)$ is a non-compact Ricci soliton with $m>0$ and $A<0$, under certain conditions, $M$ admits a non-trivial homogeneous gradient $m$-quasi Einstein ($\ric_\phi^m=Ag$) one-dimensional extension. In \cite[Theorem~1.1]{Lafuente}, Lafuente proves a converse to this result. 
\vspace{1em}


On the other hand, Chen-Liang-Zhu construct some examples of non-gradient $m$-quasi Einstein manifolds in \cite{ChenLiangZhu}. In \cite[Corollary~4.1,4.2]{ExtremalBlackHoles}, Kunduri-Lucietti study the non-gradient $m$-quasi Einstein metrics with $m=2$ in the context of vacuum, homogeneous near-horizon geometries, which gives us motivation to study non-gradient $m$-quasi Einstein metrics.
\vspace{1em}


Our main theorems give us a characterization of Lie groups which have a discrete group of isometries acting cocompactly and which satisfy $\ric_X^m=Ag$.

\begin{theorem}\label{theorem:left invariant and killing}
Let $G$ be a Lie group and let $\Gamma$ be a discrete group of isometries which acts cocompactly on $G$. Let $X$ be a vector field which is invariant under $\Gamma$. If $(G,g,X)$ satisfies $\frac{1}{2}\mathcal{L}_Xg-\frac{1}{m}X^*\otimes X^*=q$, where $q$ and $g$ are left invariant, then $X$ is left invariant. If we also assume that $\tr(q\circ ad_X)=0$, then $X$ is a Killing vector field.
\end{theorem}

Theorem \ref{theorem:left invariant and killing} was proven by Chen-Liang-Zhu in \cite[Theorem~1.1]{ChenLiangZhu} in the case when $G$ is a compact Lie group and $q=\ric$. Our next theorem gives us a characterization of the product of Einstein manifolds of any dimension which satisfy the $m$-quasi Einstein equation.

\begin{theorem}\label{theorem:compact einstein cross einstein} Consider the compact quotient of $M\times N$ with the product metric, where $M$ and $N$ are simply-connected complete Einstein manifolds. Then the only nontrivial solutions to $\ric_X^m=Ag$ occurs when either $M$ is $\mathbb{R}$ or $N$ is $\mathbb{R}$. 
\end{theorem}

We apply the results above to classify the $m$-quasi Einstein solutions for locally homogeneous 3-manifolds which admit compact quotient. 

\vspace{1em}

\begin{theorem}\label{thm: main theorem}
Let $M^3$ be a compact locally homogeneous Riemannian manifold with $\ric_X^m=Ag$. 

\begin{enumerate}
    \item If $m>0$ and $A>0$, then there exist $m$-quasi Einstein solutions if and only if $M^3$ is a compact quotient of $SU(2)$.
    \item If $m>0$ and $A=0$, then there exist solutions if and only if $M^3$ is a compact quotient of $SU(2)$ or $\mathbb{R}^3$, where the solution on $\mathbb{R}^3$ is $X=0$.
    \item If $m>0$ and $A<0$, then there exist solutions if and only if $M^3$ is a compact quotient of $SU(2)$, $Nil$, or $H^2\times \mathbb{R}$. 
    \item If $m<0$ and $A>0$, then there exist solutions if and only if $M^3$ is a compact quotient of $SU(2)$ or $S^2\times \mathbb{R}$. 
    \item If $m<0$ and $A=0$, then there exist solutions if and only if $M^3$ is a compact quotient of $\mathbb{R}^3$ or $\widetilde{SL_2(\mathbb{R})}$, where the solution on $\mathbb{R}^3$ is trivial.
    \item If $m<0$ and $A<0$, there are no $m$-quasi Einstein solutions on $M^3$ .
\end{enumerate}
\end{theorem}
\vspace{1em}

\begin{remark}
In a related result, Buttsworth studied the prescribed Ricci tensor problem on these spaces in \cite{Buttsworth}. This result when $m=2$ was also proven by Kunduri-Lucietti in \cite{ExtremalBlackHoles}.
\end{remark}


If $M^n$ is a homogeneous Einstein manifold, where $\ric=Ag$, then if $A>0$, then $M$ is compact by Myers' Theorem, if $A=0$, then $M$ is flat by Alekseevskii-Kimel'fel'd in \cite{Alekseevskii-Kimelfeld}, and if $A<0$, then $M$ is not compact by Bochner's Theorem, which can be found in Section \ref{section: Splitting theorem and myers theorem}. If we compare this to Theorem \ref{thm: main theorem}, we see that this structure does not hold for $m$-quasi Einstein metrics. When $A=0$, there exist solutions on (compact quotients of) $SU(2)$, which are not flat. Similarly, in the $A<0$ case, there exist solutions on compact quotients of $SU(2)$. 
\vspace{1em}

In \cite[Lemma~4.4]{Wylie_Warped_Product_Splitting}, we see that if $M^n$ is a compact manifold with infinite fundamental group satisfying $\ric_\phi^m=Ag$ where $A=0$, with $m=1-n<0$, then the universal cover has a warped product splitting. By Theorem \ref{thm: main theorem}, there exist solutions for the compact quotient of $\widetilde{SL_2(\mathbb{R})}$ if $M^n$ satisfies $\ric_X^m=Ag$ when $m<0$ and $A=0$. This is interesting because $\widetilde{SL_2(\mathbb{R})}$ clearly does not split.
\vspace{1em}

We organize the paper in the following way. In Section \ref{section: locally homogeneous 3-manifolds}, we give a characterization, due to Singer, of locally homogeneous 3-manifolds. We then explain our approach for the rest of the paper to compute solutions to the $m$-quasi Einstein equation. 
\vspace{1em}

In Section \ref{section: Unimodular Lie Groups}, we introduce theory which simplifies the $m$-quasi Einstein equation when $M^n$ is  a unimodular Lie group, and we compute the solutions in Section \ref{section: lie groups}.   In Section \ref{section: Splitting theorem and myers theorem}, we discuss using the $\ric_X^m$ version of Myers' Theorem and the Splitting Theorem in order to study the case when $m>0$, $A\geq 0$ as in Theorem \ref{thm: main theorem}.
\vspace{1em}

In Section \ref{section: analyze new equation}, we analyze the equation $\frac{1}{2}\mathcal{L}_Xg-\frac{1}{m}X^*\otimes X^*=\lambda g$ in order to classify the $m$-quasi Einstein equations of the locally homogeneous 3-manifolds that admit compact quotient which are not Lie groups. We also classify the nontrivial m-quasi Einstein metrics that can be the product of two Einstein metrics in Section 6. Then, we finish our classification and we also show that there are no solutions to $\ric_X^m=Ag$ on compact hyperbolic manifolds of any dimension. In Section \ref{section:table}, we give a table which summarizes our results. 

\section{Unimodular Lie Groups}\label{section: Unimodular Lie Groups}

In \cite[Theorem~1.1]{ChenLiangZhu}, Chen-Liang-Zhu proved that if $M$ is a compact Lie group with a left-invariant metric $g$, and if $X$ is a vector field on $M$ such that $\ric_X^m=Ag$ for $m\neq 0$, then $X$ is a left-invariant. Furthermore, $X$ is a Killing vector field \cite[Theorem~2.3]{ChenLiangZhu}.
\vspace{1em}

Chen-Liang-Zhu prove \cite[Theorem~1.1]{ChenLiangZhu} by first proving that $X$ is left-invariant, and then proving that $X$ is Killing using properties of the Ricci tensor. We will consider $\frac{1}{2}\mathcal{L}_Xg-\frac{1}{m}X^*\otimes X^*=q$ where $q$ is a left-invariant tensor, which is more general than $\ric+\frac{1}{2}\mathcal{L}_Xg-\frac{1}{m}X^*\otimes X^*=Ag$. Rather than considering $G$ a compact Lie group, we assume $G$ admits a discrete group of isometries, $\Gamma$, which acts cocompactly on $G$. Next, we give the definition for $ad_X$ in order to state a linear algebra fact to prove that $X$ is Killing given that $X$ is a left-invariant vector field which satisfies $\ric_X^m=Ag$.
\vspace{1em}

\begin{defn}
If $G$ is a Lie group and if $\mathfrak{g}$ is the Lie algebra of $G$, then we define $ad_X:\mathfrak{g}\rightarrow \mathfrak{g}$ by $ad_X(Y)=[X,Y]$, where $X,Y$ are vector fields in $\mathfrak{g}$. 
\end{defn}
\vspace{1em}

If $G$ is a Lie group which admits a discrete subgroup $\Gamma$ with compact quotient, then $G$ must be unimodular. It is a linear algebra fact that if $G$ is a unimodular Lie group, then there exists a basis $\{X_i\}_{i=1}^n$ of $\mathfrak{g}$, the Lie Algebra of $G$, such that $g(ad_X(X_i),X_i)=0$ for all $i$. We will use these facts about Lie groups to prove our main lemmas, which are generalizations of Chen-Liang-Zhu's \cite[Theorem~1.1]{ChenLiangZhu} and \cite[Theorem~2.3]{ChenLiangZhu}.

\begin{lemma}\label{lemma:chen_without_ricci}
Let $G$ be a connected Lie group and let $\Gamma$ be a discrete group of isometries which acts cocompactly on $G$. Let $X$ be a vector field which is invariant under $\Gamma$. If $(G,g,X)$ satisfies $\displaystyle\frac{1}{2}\mathcal{L}_Xg-\frac{1}{m}X^*\otimes X^*=q$, where $q$ and $g$ are left invariant, then $X$ is a left-invariant vector field. 
\end{lemma}
\vspace{1em}

\begin{proof}
Because $G$ is a Lie group which admits a discrete subgroup with compact quotient, $G$ must be unimodular. Let $M=\bigslant{G}{\Gamma}$ and let $\pi:G\rightarrow M$. By our discussion above, we can choose a basis, $\{X_i\}\in G$, such that $g(ad_{X}(X_i),X_i)=0$ for all $i$. Then let $X=\displaystyle\sum_{k=1}^n f_kX_k$, where $f_k:G\rightarrow \mathbb{R}$. Using the technique from \cite[Theorem~1.1]{ChenLiangZhu}, for all $i$, we get the following:

\begin{equation*}
\begin{aligned}
\displaystyle\frac{1}{2}\mathcal{L}_{X}g(X_i,X_i)-\frac{1}{m}X^*\otimes X^*(X_i,X_i)&=X_if_i+\sum_{k=1}^{n}f_k g(\nabla_{X_i}X_k,X_i )-\frac{1}{m}f_i^2\\
&=X_if_i+\sum_{k=1}^{n}f_k g([X_i,X_k],X_i )-\frac{1}{m}f_i^2\\
&=X_if_i+g(-ad_{X}(X_i),X_i )-\frac{1}{m}f_i^2\\
&=X_if_i-\displaystyle\frac{1}{m}f_i^2.
\end{aligned}
\end{equation*}
\vspace{1em}

Then, since $M$ is compact, there exists a maximum and a minimum of the function $f_i$ on $M$. Let $r$ be a point in $M$ such that $f_i(r)$ is maximal and let $s$ be a point in $M$ such that $f_i(s)$ is minimal and let $q(\pi(X_i),\pi(X_i))=\lambda_i$. Then
\begin{equation*}
\begin{aligned}
\lambda_i&=X_if_i(r)-\displaystyle\frac{1}{m}f_i^2(r)\\
&=-\displaystyle\frac{1}{m}f_i^2(r)\\
\end{aligned}
\end{equation*}
and 

\begin{equation*}
\begin{aligned}
\lambda_i&=X_if_i(s)-\displaystyle\frac{1}{m}f_i^2(s)\\
&=-\displaystyle\frac{1}{m}f_i^2(s)\\
\end{aligned}
\end{equation*}
\vspace{1em}

Then, $f_i^2(r)=f_i^2(s)=-m\lambda_i$. We will now rule out the case $f_i(r)=-f_i(s)$ in order to show that $f_i$ must be constant.
\vspace{1em}

Let $c(t)$ be an integral curve of $X_i$. Then along $\pi\circ c(t)$, $f_i'(t)-\frac{1}{m}f_i^2(t)=\lambda_i$. Solving this equation (see Lemma \ref{lemma:function}), we have that $f_i(t)=\sqrt{-\lambda_im}$, $-\sqrt{-\lambda_i m}$, $0$, or $-\sqrt{-\lambda_i m}\tanh(\frac{\sqrt{-\lambda_i m}}{m}(t+C))$.
\vspace{1em}

Assume for the sake of contradiction that $f_i(t)$ is not constant, ie $f_i(t)=-\sqrt{-\lambda_i m}\tanh(\frac{\sqrt{-\lambda_i m}}{m}(t+C))$, where $C$ is a constant. Let $\pi\circ c(t_i)$ be a sequence of points such that $t_i\rightarrow \infty$. Since $M$ is compact, there exists a subsequence of $\{\pi\circ c(t_i)\}$ which converges to a point on $M$.
\vspace{1em}

Now consider the set $\overline{\{\pi\circ c(t):t\in\mathbb{R}\}}$. Since this set is closed, $f_i$ has a maximal point, $t_{max}$ on this set. Because the supremum of the $\tanh$ function is $1$, we know that the maximum of $f_i(t)$ on $\overline{\{\pi\circ c(t):t\in\mathbb{R}\}}$ is $\sqrt{-\lambda_i m}$.
\vspace{1em}

Let $b(t)$ be an integral curve of $X_i$ such that $b(0)=c(t_{max})=\sqrt{-\lambda_i m}$. Now consider the set $\{\pi\circ b(t):t\in\mathbb{R}\}$. Along $b(t)$, $f_i(t)$ is either $\sqrt{-\lambda_i m}$ or $-\sqrt{-\lambda_i m}\tanh(\frac{\sqrt{-\lambda_i m}}{m}(t+C))$. Since the supremum of $f_i(t)$ on $\{\pi\circ b(t):t\in\mathbb{R}\}$ is $\sqrt{-\lambda_i m}$ and $\tanh$ never achieves its maximum on its domain, $f_i(t)$ must be constantly $\sqrt{-\lambda_i m}$ on the set $\{\pi\circ b(t):t\in \mathbb{R}\}$.
\vspace{1em}

Finally, since $\overline{\{\pi\circ b(t):t\in\mathbb{R}\}}=\overline{\{\pi\circ c(t):t\in\mathbb{R}\}}$, $f_i(t)$ is constant on $\overline{\{\pi\circ c(t):t\in\mathbb{R}\}}$. Then, since  $f_i(t)$ is constant along every integral curve and since $G$ is connected, $f_i(t)$ is constant.

\end{proof}

\begin{lemma}\label{lemma:function}
Let $f'(t)-\frac{1}{m}f^2(t)=\lambda$, where $f:\mathbb{R}\rightarrow\mathbb{R}$ is defined for all $t$ in $\mathbb{R}$ and $\lambda$ and $m$ are constants. Then:
\begin{enumerate}
    \item If $\lambda =0$, then $f(t)=0$.
    \item If $\lambda m>0$, then there are no solutions.
    \item If $\lambda m<0$, then $f(t)=\pm\sqrt{-\lambda m}$ or $\sqrt{-\lambda m}\tanh\bigg(\frac{\sqrt{-\lambda m}}{m}(t+C)\bigg)$.
\end{enumerate}

\end{lemma}

\begin{proof}
Suppose $\lambda=0$. Then it is clear that $f(t)=0$ is a solution. If $f(0)$ is not $0$, then

\begin{align*}
f'(t)&=\frac{f(t)^2}{m}\\
\Rightarrow f(t)&=\frac{1}{C-\frac{t}{m}}
\end{align*}

where $C$ is any real number. However, at $t=mC$, $t$ blows up, which is a contradiction since $f$ has to exist for all time.
\vspace{1em}

If $\lambda m>0$, then 

\begin{equation*}
f'(t)=\frac{f(t)^2}{m}+\lambda.
\end{equation*}

Here, we see that $\displaystyle\frac{f(t)^2}{m}+\lambda$ is never zero since $\lambda m>0$. Integrating and rearranging, we get

\begin{align*}
\int \frac{f'(t)}{\frac{f^2(t)}{m}+\lambda}dt&=\int 1dt\\
\Rightarrow \frac{m}{\lambda}\int\frac{f'(t)}{1+\big(\frac{f(t)}{\sqrt{\lambda m}}\big)^2}dt&=t+C\\
\Rightarrow \sqrt{\frac{m}{\lambda}}\tan^{-1}\bigg(\frac{f(t)}{\sqrt{\lambda m}}\bigg)&=t+C,\\
\end{align*}

so then, $$f(t)=\sqrt{\lambda m}\tan\bigg(\sqrt{\frac{\lambda}{m}}(t+C)\bigg).$$
Since the $\tan$ function does not exist everywhere, $f(t)$ also does not exist everywhere. Thus, if $\lambda m>0$, there are no solutions.
\vspace{1em}

If $\lambda m<0$, then clearly $f(t)=\pm\sqrt{-\lambda m}$ is a solution to the equation. Assume $f(0)$ is not $\pm\sqrt{-\lambda m}$. Then we integrate and rearrange as follows:

\begin{align*}
\int \frac{f'(t)}{\frac{f^2(t)}{m}+\lambda}dt&=\int 1dt\\
\frac{m}{2\sqrt{-\lambda m}}\ln\bigg|\frac{1-\frac{f(t)}{\sqrt{-\lambda m}}}{1+\frac{f(t)}{\sqrt{-\lambda m}}}\bigg|&=t+C\\
\Rightarrow \bigg|\frac{1-\frac{f(t)}{\sqrt{-\lambda m}}}{1+\frac{f(t)}{\sqrt{-\lambda m}}}\bigg|&=e^{2\frac{\sqrt{-\lambda m}}{m}(t+C)}.
\end{align*}

If $\displaystyle\frac{1-\frac{f(t)}{\sqrt{-\lambda m}}}{1+\frac{f(t)}{\sqrt{-\lambda m}}}=e^{2\frac{\sqrt{-\lambda m}}{m}(t+C)}$, then $$f(t)=\sqrt{-\lambda m}\bigg(\frac{1-e^{2\frac{\sqrt{-\lambda m}}{m}(t+C)}}{1+e^{2\frac{\sqrt{-\lambda m}}{m}(t+C)}}\bigg)=\sqrt{-\lambda m}\tanh\bigg(\frac{\sqrt{-\lambda m}}{m}(t+C)\bigg).$$

If $\displaystyle\frac{1-\frac{f(t)}{\sqrt{-\lambda m}}}{1+\frac{f(t)}{\sqrt{-\lambda m}}}=-e^{2\frac{\sqrt{-\lambda m}}{m}(t+C)}$, then $f(t)=\sqrt{-\lambda m}\displaystyle\bigg(\frac{1+e^{2\frac{\sqrt{-\lambda m}}{m}(t+C)}}{1-e^{2\frac{\sqrt{-\lambda m}}{m}(t+C)}}\bigg)$. In this case, at $t=-C$, $f(t)$ does not exist, which is a contradiction.
\end{proof}

\begin{lemma}\label{lemma: chen without ricci killing}
Let $G$ be a unimodular Lie group with left-invariant metric, $g$. If $X$ is left-invariant, $\tr(q\circ ad_X)=0$, and $\displaystyle\frac{1}{2}\mathcal{L}_Xg-\frac{1}{m}X^*\otimes X^*=q$, where $q$ is left-invariant, then $X$ is Killing.
\end{lemma}
\vspace{1em}

\begin{proof}[Proof of Lemma \ref{lemma: chen without ricci killing}]
Let $\{X_i\}$ be an orthonormal basis relative to $g$ and let $X=a_1X_1+a_2X_2+...+a_nX_n$. Then, plugging in $(X_i,X_j)$ into $q=\displaystyle\frac{1}{2}\mathcal{L}_Xg-\frac{1}{m}X^*\otimes X^*$, we get

\begin{equation*}
q(X_i,X_j)=\displaystyle\frac{1}{2}\big(g([X_i,X],X_j)+g([X_j,X],X_i)-\frac{1}{m}g(X,X_i)g(X,X_j).
\end{equation*}

We denote the projection of $X_i$ onto $X$, as $\proj_{X}X_i$. Since $\proj_{X}X_i=\displaystyle\frac{g(X,X_i)X}{|X|^2}$ and $ad_X(X_i)=[X,X_i]$, we have the following: 

\begin{equation*}
q(X_i,X_j)=\displaystyle\frac{1}{2}\big(g(ad_X(X_i),X_j)+g(ad_X(X_j),X_i)\big)-\frac{|X|^2}{m}g(\proj_XX_i,X_j).
\end{equation*}

Thus, we have the following equation, where we view $q$, $ad_X$, and $\proj_X$ as matrices:

\begin{equation*}
q=\displaystyle\frac{1}{2}\big(ad_X+ad_X^T\big)-\frac{|X|^2}{m}\proj_X.
\end{equation*}

We denote ``$\cdot$" as the matrix multiplication symbol. Multiplying both sides by the matrix, $ad_X$, we get:

\begin{equation*}
\begin{aligned}
q \cdot ad_X&=\displaystyle\frac{1}{2}\big(ad_X+ad_X^T\big)\cdot ad_X-\frac{|X|^2}{m}\proj_X\cdot ad_X\\
&=\displaystyle\frac{1}{2}\big(ad_X^2+ad_X^T\cdot ad_X\big)-\frac{|X|^2}{m}\proj_X\cdot ad_X.
\end{aligned}
\end{equation*}

Taking the trace of both sides, we get

\begin{equation*}
\begin{aligned}
\tr(q \cdot ad_X)&=\displaystyle\frac{1}{2} \tr\big(ad_X^2+ad_X^T\cdot ad_X\big)-\frac{|X|^2}{m} \tr(\proj_X\cdot ad_X).\\
\end{aligned}
\end{equation*}

Then, since $\tr(q\cdot ad_X)=0$ and using that for any $n\times n$ matrix $A$, $\tr(A^2)=\tr((A^T)^2)$, we get

\begin{equation*}
0=\displaystyle\frac{1}{4} \tr\big((ad_X+ad_X^T)^2)-\frac{|X|^2}{m} \tr(\proj_X\cdot ad_X).
\end{equation*}

Now, plugging in $X_i$, one of the orthonormal basis vectors into $ad_X\cdot \proj_X$ and using that $\tr(AB)=\tr(BA)$ for any two matrices $A$ and $B$, we get:

\begin{equation*}
\begin{aligned}
ad_X\cdot \proj_X(X_i)&=\frac{a_i}{|X|^2}[X,X]\\
&=0.
\end{aligned}
\end{equation*}

Thus, we have 
$0=\displaystyle\frac{1}{4} \tr\big((ad_X+ad_X^T)^2)$.
\vspace{1em}

Now, since $ad_X+ad_X^T$ is symmetric, we can diagonalize $ad_X+ad_X^T$, and call the diagonalized matrix $D$. Then, $\tr((ad_X+ad_X^T)^2)=\tr(D^2)$. Since the eigenvalues in $D^2$ are nonnegative and $\tr(D^2)$ is the sum of the eigenenvalues of $D^2$, $\displaystyle\frac{1}{2}(ad_X+ad_X^T)=0$. Thus, $X$ is Killing.
\end{proof}



Next, we will apply Lemma \ref{lemma:chen_without_ricci} to metrics which satisfy $\ric_X^m=Ag$.
\vspace{1em}

\begin{theorem}\label{theorem: ric_X^m and chen without ricci}
Let $G$ be a Lie group and let $\Gamma$ be a discrete group of isometries which acts cocompactly on $G$, where $\pi:G\rightarrow \bigslant{G}{\Gamma}$ is a covering map. If $(\bigslant{G}{\Gamma},g,X)$ satisfies $\ric_X^m=Ag$, then $\widetilde{X}=\pi^*(X)$ is left invariant and Killing.
\end{theorem}
\vspace{1em}

\begin{proof}
First, we let $\widetilde{g}=\pi^*(g)$, be the pullback metric of $g$. Since $\pi$ is a local isometry, $\ric_{\widetilde{X}}^m=A\widetilde{g}$
\vspace{1em}


Since $A\widetilde{g}-\ric_{\widetilde{g}}$ is left-invariant, by Lemmas \ref{lemma:chen_without_ricci} and \ref{lemma: chen without ricci killing}, $\widetilde{X}$ is left-invariant and Killing.
\end{proof}
\vspace{1em}

We immediately get the following corollary, which we will use throughout Section \ref{section: lie groups}.

\begin{corollary}\label{cor:Unimodular leftinvariant killingfield}
If $M^n$ is a unimodular Lie Group and if $\ric_X^m=Ag$ with $X$ a left-invariant vector field and $g$ a left-invariant metric, then $X$ is a Killing field.
\end{corollary}
\vspace{1em}

\begin{lemma}\label{lemma: aiaj=0}
Suppose $(M^n,g)$ is a Lie group which satisfies $\ric_X^m=Ag$ where  $X$ is nonzero, left-invariant, and Killing. If $\{X_1, X_2,... X_n\}$ is an eigenbasis of the Ricci tensor of left invariant fields, then $X$ is a multiple of one of the eigenbasis vectors (ie there exists $1\leq m\leq n$ such that $X=a_mX_m$).
\end{lemma}

\begin{proof}
Since $X$ is left-invariant and Killing, we have for all $1\leq i,j\leq n$ where $i\neq j$,

\begin{equation*}
    \ric_X^m(X_i,X_j)=-\frac{1}{m}a_ia_j.
\end{equation*}
\vspace{1em}

Now $\ric_X^m(X_i,X_j)=Ag(X_i,X_j)=0$ for all sets of $i,j$ if and only if at least $n-1$ sets of $a_k$ are $0$. Thus, $X=a_mX_m$ for some $1\leq m\leq n$. 
\end{proof}

\section{Preliminaries About Locally Homogeneous 3-Manifolds}\label{section: locally homogeneous 3-manifolds}

In this section, we will discuss locally homogeneous three-manifolds, which we will use to prove our main results. We first give definitions of locally homogeneous and homogeneous, which can be found in \cite{IsenbergJackson}.

\begin{defn}
Let $(M,g)$ be a Riemmanian manifold. Then $(M,g)$ is locally homogeneous if for every pair of points $x,y\in M$, there exists neighborhoods $U_x$ of $x$ and $V_y$ of $y$ such that there is an isometry $\psi$ mapping $(U_x,g|_{U_x})$ to $(V_y,g|_{V_y})$, with $\psi(x)=y$.
\end{defn}

\begin{defn}
Let $(M,g)$ be a Riemmanian manifold. Then $(M,g)$ is homogeneous if for every pair of points $x,y\in M$, there exists an isometry $\psi$, $\psi(x)=y$.
\end{defn}

According to Singer in \cite{Singer}, for every locally homogeneous geometry $(M^3,g)$, the universal cover, $(\widetilde{M^3},\widetilde{g})$, is homogeneous. If $(\widetilde{M}^3,\widetilde{g})$ is a homogeneous, simply connected manifold that admits a compact quotient, then it is one of the following: $\mathbb{R}^3$, $SU(2)$, $\widetilde{SL_2(\mathbb{R})}$, $Nil$, $E(1,1)$, $E(2)$, $H^3$, $S^2\times \mathbb{R}$, or $H^2\times \mathbb{R}$ \cite[Table~1]{IsenbergJackson}. 
\vspace{1em}

Since $\widetilde{X}$ is a left-invariant solution to $\ric_{\widetilde{X}}^m=A\widetilde{g}$ if and only if $d\pi(\widetilde{X})$ is a solution to $\ric_X^m=Ag$, where $\pi:\widetilde{M}\rightarrow M$ is the universal covering map, we study these nine geometries in order to classify $m$-quasi Einstein metrics on locally homogeneous three manifolds. Of the nine geometries, $\mathbb{R}^3$, $SU(2)$, $\widetilde{SL_2(\mathbb{R})}$, $Nil$, $E(1,1)$, and $E(2)$ are Lie groups. We can also use that $H^2$ is a Lie group to study $H^2\times\mathbb{R}$. We will explicitly calculate the metrics on the Lie groups which satisfy $\ric_X^m=Ag$ using the methods of Section \ref{section: Unimodular Lie Groups}. We will study the equation $\frac{1}{2}\mathcal{L}_Xg-\frac{1}{m}X^*\otimes X^*=\lambda g$ in order to calculate the $m$-quasi Einstein metrics on $S^2\times \mathbb{R}$ and $H^3$.
\vspace{1em}
%

Throughout this paper, we will use the following computations by Milnor:
\begin{lemma}\cite[pages~305, 307]{Milnor}\label{lemma:Milnor}
Let $G$ be a 3-dimensional unimodular Lie group with left invariant metric. If $L$ is self-adjoint, then there exists an orthonormal basis $\{X_1,X_2,X_3\}$  consisting of eigenvectors $LX_i=\lambda_i^* X_i$. We obtain the following:

$$[X_2, X_3]=\lambda_1^*X_1$$
$$[X_3, X_1]=\lambda_2^*X_2$$
$$[X_1, X_2]=\lambda_3^*X_3.$$
\vspace{1em}

The following chart gives us the signs of $\lambda_i^*$ for $SU(2)$, $\widetilde{SL_2(\mathbb{R})}$, $E(2)$, $E(1,1)$, $Nil$, and $\mathbb{R}^3$.

\begin{table}[ht]
\bgroup
\def\arraystretch{1.5}
\begin{center}
\begin{tabular}{ c c c c }
Lie Group & $\lambda_1^*$ & $\lambda_2^*$ & $\lambda_3^*$ \\
\hline
$Nil$ & $\lambda_1^*>0$ & $\lambda_2^*=0$ & $\lambda_3^*=0$\\
$\widetilde{SL_2(\mathbb{R})}$ & $\lambda_1^*>0$ & $\lambda_2^*>0$ & $\lambda_3^*<0$\\
$E(1,1)$ & $\lambda_1^*>0$ & $\lambda_2^*<0$ & $\lambda_3^*=0$\\
$E(2)$ & $\lambda_1^*>0$ & $\lambda_2^*>0$ & $\lambda_3^*=0$\\
$\mathbb{R}^3$ & $\lambda_1^*=0$ & $\lambda_2^*=0$ & $\lambda_3^*=0$\\
$SU(2)$ & $\lambda_1^*>0$ & $\lambda_2^*>0$ & $\lambda_3^*>0$
\end{tabular}
\end{center}
\egroup

\caption{}
\label{table:lie bracket}
\end{table}
\vspace{1em}

From now on, let $\lambda_i=|\lambda_i^*|$. 
\end{lemma}
\vspace{1em}

Because we will be using that $X$ is Killing for unimodular Lie groups with $\ric_X^m=Ag$, it will be useful to calculate $\mathcal{L}_Xg$. 

\begin{prop}\label{prop:lie derivative computation}
Let $X=a_1X_1+a_2X_2+a_3X_3$ be left-invariant vector field on a 3-dimensional unimodular Lie group with left invariant metric. Then using the same notation as in Lemma \ref{lemma:Milnor}, we have the following:

\begin{align*}
&\mathcal{L}_Xg(X_i,X_i)=0\text{ for all }i\\[0.5ex]  
&\mathcal{L}_Xg(X_1,X_2)=-a_3\lambda_2^*+a_3\lambda_1^*\\[0.5ex]
&\mathcal{L}_Xg(X_1,X_3)=-a_2\lambda_1^*+a_2\lambda_3^*\\[0.5ex]
&\mathcal{L}_Xg(X_2,X_3)=-a_1\lambda_3^*+a_1\lambda_2^*
\end{align*}

%
%
%
\end{prop}
\vspace{1em}
\begin{proof}

We have the following computation for $\mathcal{L}_Xg$:
\vspace{1em}

$\mathcal{L}_Xg(X_i,X_j)$
\begin{equation*}
\begin{aligned}
&=g(\nabla_{X_i}(a_1X_1+a_2X_2+a_3X_3),X_j)+g(\nabla_{X_j}(a_1X_1+a_2X_2+a_3X_3),X_i)\\
&=\displaystyle\sum_k a_kg(\nabla_{X_i}X_k,X_j)+a_kg(\nabla_{X_j}X_k,X_i)\\
&=\displaystyle\sum_k g(\nabla_{X_k}X_i+[X_i,X_k],X_j)+g(\nabla_{X_k}X_j+[X_j,X_k],X_i)\\
&=\sum_k a_kg([X_i,X_k],X_j)+a_kg([X_j,X_k],X_i)+D{X_k}g(X_i,X_j)\\
&=\sum_k a_kg([X_i,X_k],X_j)+a_kg([X_j,X_k],X_i).
\end{aligned}
\end{equation*}
\vspace{1em}

Then, using Lemma \ref{lemma:Milnor}, we get:

\begin{align*}
&\mathcal{L}_Xg(X_i,X_i)=0\text{ for all }i\\[0.5ex]  
&\mathcal{L}_Xg(X_1,X_2)=-a_3\lambda_2^*+a_3\lambda_1^*\\[0.5ex]
&\mathcal{L}_Xg(X_1,X_3)=-a_2\lambda_1^*+a_2\lambda_3^*\\[0.5ex]
&\mathcal{L}_Xg(X_2,X_3)=-a_1\lambda_3^*+a_1\lambda_2^*
\end{align*}

%
%
%
\end{proof}
\vspace{1em}

Finally, we recall the definition of the Ricci quadratic form, $r(x)$, as introduced by Milnor in \cite{Milnor}, and the signatures of the Ricci forms of $Nil$, $E(1,1)$, $\widetilde{SL_2(\mathbb{R})}$, $E(2)$, $\mathbb{R}^3$, and $SU(2)$ when the metric is left invariant. 

\begin{defn}
The Ricci quadratic form, $r(X)$ takes vectors $X\in TM$ to $\mathbb{R}$ and is defined as follows:
 $$g(r(X),Y)=\ric(X,Y)$$ for all $Y\in TM$.
\vspace{1em}

The collection of signs of $r(e_i)$, namely, $\{sign(r(e_i))\}_{i=1}^n$, is called the signature of the quadratic form $r$, where $\{e_i\}_{i=1}^n$ is any orthonormal basis for the tangent space.
\end{defn}

\begin{table}[ht]
\bgroup
\def\arraystretch{1.5}
\begin{center}
\begin{tabular}{ l|c|c|c|r }
\text{Lie Group} & $r(e_1)$ & $r(e_2)$ & $r(e_3)$ & \text{Reference}\\
\hline
$Nil$ & $r(e_1)>0$ & $r(e_2)<0$ & $r(e_3)<0$ & \text{\cite[Corollary~4.6]{Milnor}}\\
$E(1,1)$, $\widetilde{SL_2(\mathbb{R})}$ & $r(e_1)>0$ & $r(e_2)<0$ & $r(e_3)<0$ & \\
& $r(e_1)=0$ & $r(e_2)=0$ & $r(e_3)<0$ & \text{ \cite[Corollary~4.7]{Milnor}}\\
$E(2)$ & $r(e_1)>0$ & $r(e_2)<0$ & $r(e_3)<0$ & \text{\cite[Corollary~4.8]{Milnor}}\\
$\mathbb{R}^3$ & $r(e_1)=0$ & $r(e_2)=0$ & $r(e_3)<0$ & \\
$SU(2)$ & $r(e_1)>0$ & $r(e_2)>0$ & $r(e_3)>0$ & \\
& $r(e_1)>0$ & $r(e_2)=0$ & $r(e_3)=0$ & \\
& $r(e_1)>0$ & $r(e_2)<0$ & $r(e_3)<0$ & \text{\cite[Corollary~4.5]{Milnor}}
\end{tabular}
\end{center}
\egroup
\caption{}
\label{table: ricci signature}
\end{table}
\vspace{1em}

\section[]{$m$-Quasi Einstein Solutions for $Nil$, $\widetilde{SL_2\mathbb{R}}$, $E(1,1)$, $E(2)$ and $H^2\times\mathbb{R}$}\label{section: lie groups}

In this section, we will compute solutions to the $m$-quasi Einstein equation for the Lie groups $Nil$, $\widetilde{SL_2(\mathbb{R})}$, $E(1,1)$, and $E(2)$. We will also compute solutions to $H^2\times \mathbb{R}$, using the Lie group structure of $H^2$. 
\vspace{1em}
%

We will use Tables \ref{table:lie bracket} and \ref{table: ricci signature} as well as the next remark
to find examples of $X$ which gives us $\ric_X^m=Ag$ for $m>0$ and $A<0$ for the space $Nil$.
\vspace{1em}

\begin{remark}\label{remark:Nil}
By \cite[Corollary~4.5]{Milnor}, for any left invariant metric on $Nil$, the principal Ricci curvatures satisfy $|r(e_1)|=|r(e_2)|=|r(e_3)|=|\rho|$. 
\end{remark}

\begin{prop}\label{prop: nil A<0 m>0}
Consider $Nil$ with $\ric_X^m=Ag$. If $g$ is a left-invariant metric and if $X$ is a left-invariant vector field, then there exist examples of $X$ such that $\ric_X^m=Ag$ if and only if $A<0$ and $m>0$.
\end{prop}

\begin{proof}
Let $\{X_1,X_2,X_3\}$ be an orthonormal basis where $\ric(X_1,X_1)=\rho$, $\ric(X_2,X_2)=-\rho$, and $\ric(X_3,X_3)=-\rho$ as in Table \ref{table: ricci signature} and Remark \ref{remark:Nil}.
Let $X=a_1X_1+a_2X_2+a_3X_3$ where $a_1$, $a_2$, and $a_3$ are all constants. By Corollary \ref{cor:Unimodular leftinvariant killingfield}, $X$ is a Killing field so we set $\mathcal{L}_Xg(X_i,X_j)=0$ for all $i,j=1,2,3$ as follows: 

\begin{align*}
   &\mathcal{L}_Xg(X_1,X_2)=a_3\lambda_1=0\\[0.5ex]
   &\mathcal{L}_Xg(X_1,X_3)=-a_2\lambda_1=0 
\end{align*}
\vspace{1em}

%

where every other combination of $\mathcal{L}_Xg(X_i,X_j)$ is zero by definition of $Nil$. Thus, $a_2=a_3=0$. We compute $\ric_X^m$ as follows:
\vspace{1em}

\begin{align*}
\ric_X^m(X_1,X_1)&=\rho-\frac{1}{m}a_1^2\\
\ric_X^m(X_2,X_2)&=-\rho-\frac{1}{m}a_2^2=-\rho\\
\ric_X^m(X_3,X_3)&=-\rho-\frac{1}{m}a_3^2=-\rho
\end{align*}
\vspace{1em}
%
%
%
%
%
%

Thus, $\ric_X^m=Ag$ if and only if $X=\pm\sqrt{2m\rho}X_1$. In this case, $m>0$ and $A=-\rho<0$.

\end{proof}
\vspace{1em}

Now, we will find examples of $X$ which satisfy $\ric_X^m=Ag$ for the spaces $E(1,1)$ and $\widetilde{SL_2(\mathbb{R})}$. 
\vspace{1em}
%

\begin{prop}
Consider $\widetilde{SL_2(\mathbb{R})}$. If $g$ is a left-invariant metric and if $X$ is a left-invariant vector field, then there exist examples of $\ric_X^m=Ag$ if and only if $m<0$ and $A=0$.
\end{prop}

\begin{proof}
Let $g$ is a left-invariant metric and let $X$ be a left-invariant vector field, where $X=a_1X_1+a_2X_2+a_3X_3$ with $\{X_1,X_2,X_3\}$ an orthonormal basis. By Corollary \ref{cor:Unimodular leftinvariant killingfield}, $X$ must be a Killing field if $\ric_X^m=Ag$, so we set $\mathcal{L}_Xg(X_i,X_j)=0$ for all $i,j=1,2,3$ as follows:

\begin{align*}
&\mathcal{L}_Xg(X_1,X_2)=a_3(\lambda_1-\lambda_2)=0\\[0.5ex]
&\mathcal{L}_Xg(X_1,X_3)=a_2(-\lambda_1-\lambda_3)=0\\[0.5ex]
&\mathcal{L}_Xg(X_2,X_3)=a_1(\lambda_2+\lambda_3)=0
\end{align*}
\vspace{1em}
%
%
%

where all other pairs of $\mathcal{L}_Xg(X_i,X_j)=0$ by properties of $\widetilde{SL_2(\mathbb{R})}$. By the above, we must have $a_1=a_2=0$ and either $a_3=0$ or $\lambda_1=\lambda_2$.
\vspace{1em}

By Table \ref{table: ricci signature}
, the signature for the Ricci form is $(+,-,-)$ or $(0,0,-)$. 
\vspace{1em}

If the Ricci form is $(+,-,-)$, let $|\ric(X_i,X_i)|=\rho_i$. Then, plugging in $(X_i,X_j)$, where $i,j=1,2,3$ into $\ric_X^m=Ag$, we get the following set of equations:
\vspace{1em}

\begin{align*}
\ric_X^m(X_1,X_1)&=\rho_1-\frac{1}{m}a_1^2=\rho_1\\
\ric_X^m(X_2,X_2)&=-\rho_2-\frac{1}{m}a_2^2=-\rho_2\\
\ric_X^m(X_3,X_3)&=-\rho_3-\frac{1}{m}a_3^2
\end{align*}
\vspace{1em}
%
%
%
%
%
%

In this case, we cannot have $\ric_X^m=Ag$ since $\ric_X^m(X_1,X_1)>0$ and $\ric_X^m(X_2,X_2)<0$.
\vspace{1em}

If the Ricci form is $(0,0,-)$, then we get the following set of equations:
\vspace{1em}

\begin{align*}
\ric_X^m(X_1,X_1)&=-\frac{1}{m}a_1^2=0\\
\ric_X^m(X_2,X_2)&=-\frac{1}{m}a_2^2=0\\
\ric_X^m(X_3,X_3)&=-\rho_3-\frac{1}{m}a_3^2
\end{align*}
\vspace{1em}
%
%
%
%
%
%

Then,  $\ric_X^m=Ag$ if and only if $a_3=\sqrt{-m\rho_3}$, $A=0$, and $m<0$.
\end{proof}
\vspace{1em}

\begin{prop}
Consider $E(1,1)$. If $g$ is a left-invariant metric and if $X$ is a left-invariant vector field, then there are no solutions to $\ric_X^m=Ag$.
\end{prop}

\begin{proof}
Let $g$ is a left-invariant metric and let $X$ be a left-invariant vector field, where $X=a_1X_1+a_2X_2+a_3X_3$ with $\{X_1,X_2,X_3\}$ an orthonormal basis. By Corollary \ref{cor:Unimodular leftinvariant killingfield}, $X$ must be a Killing field if $\ric_X^m=Ag$, so we set $\mathcal{L}_Xg(X_i,X_j)=0$ for all $i,j=1,2,3$ as follows:

\begin{align*}
&\mathcal{L}_Xg(X_1,X_2)=a_3(\lambda_2+\lambda_1)=0\\[0.5ex]
&\mathcal{L}_Xg(X_1,X_3)=-a_1\lambda_2=0\\[0.5ex]
&\mathcal{L}_Xg(X_2,X_3)=-a_2\lambda_1=0
\end{align*}
\vspace{1em}
%
%
%

All other $\mathcal{L}_Xg(X_i,X_j)=0$ by properties of $E(1,1)$. By the three equations above, $a_1=a_2=a_3=0$. By Table \ref{table: ricci signature}
, the signature for the Ricci form is $(+,-,-)$ or $(0,0,-)$. If the Ricci form is $(+,-,-)$, let $|\ric(X_i,X_i)|=\rho_i$. Then, plugging in all iterations of $(X_i,X_j)$, $i,j=1,2,3$, we get the following:
\vspace{1em}

\begin{align*}
\ric_X^m(X_1,X_1)&=\rho_1-\frac{1}{m}a_1^2=\rho_1\\
\ric_X^m(X_2,X_2)&=-\rho_2-\frac{1}{m}a_2^2=-\rho_2\\
\ric_X^m(X_3,X_3)&=-\rho_3-\frac{1}{m}a_3^2=-\rho_3
\end{align*}
\vspace{1em}
%
%
%
%
%
%

$\ric_X^m$ cannot equal $Ag$ since $\ric_X^m(X_1,X_1)>0$ and $\ric_X^m(X_2,X_2)<0$.
\vspace{1em}

If the Ricci form is $(0,0,-)$, then we get the following set of equations:
\vspace{1em}

\begin{align*}
\ric_X^m(X_1,X_1)&=-\frac{1}{m}a_1^2=0\\
\ric_X^m(X_2,X_2)&=-\frac{1}{m}a_2^2=0\\
\ric_X^m(X_3,X_3)&=-\rho_3-\frac{1}{m}a_3^2
\end{align*}
\vspace{1em}
%
%
%
%
%
%

In this case, we cannot have $\ric_X^m=Ag$ since $\ric_X^m(X_1,X_1)=\ric_X^m(X_2,X_2)=0$ and $\ric_X^m(X_3,X_3)<0$.
\end{proof}
\vspace{1em}

Finally, we will find that there are no examples of $X$ on $E(2)$ which give us $\ric_X^m=Ag$. 
\vspace{1em}
%

\begin{prop}
Consider $E(2)$. If $g$ is a left-invariant metric and if $X$ is a left-invariant vector field, then there are no solutions to $\ric_X^m=Ag$.
\end{prop}

\begin{proof}
Let $g$ is a left-invariant metric and let $X$ be a left-invariant vector field, where $X=a_1X_1+a_2X_2+a_3X_3$ with $\{X_1,X_2,X_3\}$ an orthonormal basis. By Corollary \ref{cor:Unimodular leftinvariant killingfield}, $X$ must be a Killing field if $\ric_X^m=Ag$, so we set $\mathcal{L}_Xg(X_i,X_j)=0$ for all $i,j=1,2,3$ as follows:

\begin{align*}
&\mathcal{L}_Xg(X_1,X_2)=a_3(\lambda_1-\lambda_2)=0\\[0.5ex]
&\mathcal{L}_Xg(X_1,X_3)=-a_2\lambda_1=0\\[0.5ex]
&\mathcal{L}_Xg(X_2,X_3)=a_1\lambda_2=0
\end{align*}
\vspace{1em}
%
%

All other $\mathcal{L}_Xg(X_i,X_j)=0$ by properties of $E(2)$. By the three equations above, $a_1=a_2=0$ and either $\lambda_1=\lambda_2$ or $a_3=0$. By Table \ref{table: ricci signature}
, the signature for the Ricci form is $(+,-,-)$. Letting $|\ric(X_i,X_i)|=\rho_i$, we plug in all iterations of $(X_i,X_j)$, $i,j=1,2,3$ as follows:
\vspace{1em}

\begin{align*}
\ric_X^m(X_1,X_1)&=\rho_1-\frac{1}{m}a_1^2=\rho_1\\
\ric_X^m(X_2,X_2)&=-\rho_2-\frac{1}{m}a_2^2\\
\ric_X^m(X_3,X_3)&=-\rho_3-\frac{1}{m}a_3^2
\end{align*}
\vspace{1em}
%
%
%
%
%
%

$\ric_X^m$ cannot equal $Ag$ since $\ric_X^m(X_1,X_1)>0$ and $\ric_X^m(X_2,X_2)<0$.
\vspace{1em}
\end{proof}

\begin{prop}\label{prop:r3}
Consider $\mathbb{R}^3$. If $g$ is a left-invariant metric and if $X$ is a left-invariant vector field, then the only solutions of $\ric_X^m=Ag$ occur when $m\neq 0$, $A=0$, and $X=0$.
\end{prop}

\begin{proof}
Let $g$ is a left-invariant metric and let $X$ be a left-invariant vector field, where $X=a_1X_1+a_2X_2+a_3X_3$ with $\{X_1,X_2,X_3\}$ an orthonormal basis of left-invariant vector fields. By Corollary \ref{cor:Unimodular leftinvariant killingfield},
$X$ must be a Killing field if $\ric_X^m=Ag$. By \cite[page~307]{Milnor}, $\mathcal{L}_Xg(X_i,X_j)=0$ for all $i,j=1,2,3$ and $\ric(X_i,X_j)=0$ for all $i,j=1,2,3$, so we have the following sets of equations for $\ric_X^m(X_i,X_j)$.
\vspace{1em}

\begin{align*}
\ric_X^m(X_1,X_1)&=-\frac{1}{m}a_1^2\\
\ric_X^m(X_2,X_2)&=-\frac{1}{m}a_2^2\\
\ric_X^m(X_3,X_3)&=-\frac{1}{m}a_3^2
\end{align*}
\vspace{1em}
%
%
%
%
%
%

Setting $\ric_X^m=Ag$, the only solutions are when $m\neq 0$, $A=0$, and $X=0$.
\end{proof}
\vspace{1em}

\begin{remark} 
Since $\mathbb{R}^3$ is Ricci flat, Proposition \ref{prop:r3} also follows from Proposition \ref{prop: lambda_m<0 global condition}.
\end{remark}

\begin{prop}
If $g$ is a left-invariant metric on $H^2\times \mathbb{R}$ and if $X$ is a left-invariant vector field then there exist solutions to $\ric_X^m=Ag$ if and only if $A<0$ and $m>0$.
\end{prop}

\begin{proof}
Let $\{X_1,X_2,\frac{\partial}{\partial r}\}$ be an orthonormal basis where $\{X_1,X_2\}$ are in $TH^2$ and $\frac{\partial}{\partial r}$ is in $T\mathbb{R}$. Let $X=a_1X_1+a_2X_2+a_3\frac{\partial}{\partial r}$. We compute the Lie derivatives as follows: 

\begin{align*}
&\mathcal{L}_Xg(X_1,X_1)=2g(\nabla_{X_1}X,X_1)=2g(-a_2X_2,X_1)=0\\[0.5ex]
&\mathcal{L}_Xg(X_2,X_2)=2g(\nabla_{X_2}X,X_2)=2g(-a_1X_2+a_2X_1,X_2)=-2a_1\\[0.5ex]
&\mathcal{L}_Xg (\textstyle \frac{\partial}{\partial r},\frac{\partial}{\partial r})=0\\[0.5ex]
&\mathcal{L}_Xg(X_1,X_2)=g(\nabla_{X_1}X,X_1)+g(\nabla_{X_1}X,X_1)=g(-a_1X_2+a_2X_1,X_1)=a_2\\[0.5ex]
&\mathcal{L}_Xg (X_2, \textstyle\frac{\partial}{\partial r} )=g(\nabla_{X_2}X, \frac{\partial}{\partial r})+g(\nabla_{\textstyle\frac{\partial}{\partial r}}X,X_2)=0
\end{align*}
\vspace{1em}

By Corollary \ref{cor:Unimodular leftinvariant killingfield}, $X$ must be a Killing field, so we set $\mathcal{L}_Xg=0$ to get that $a_1=a_2=0$. We have that $\ric(X_1,X_1)=\ric(X_2,X_2)=-\rho g$ where $\rho>0$, and $\ric(\frac{\partial}{\partial r},\frac{\partial}{\partial r})=0$, so we can compute $\ric_X^m$ as follows:

\begin{align*}
\ric_X^m(X_1,X_1)&=-\rho\\
\ric_X^m(X_2,X_2)&=-\rho\\
\ric_X^m(\frac{\partial}{\partial r},\frac{\partial}{\partial r})&=-\frac{1}{m}a_3^2\\
\end{align*}
\vspace{1em}

%
Thus, $\ric_X^m=Ag$ if and only if $X=\pm\sqrt{\rho m}\frac{\partial}{\partial r}$, where $A=-\rho<0$ and $m>0$.

\end{proof}

We will show that we can find examples of $X$ such that $\ric_X^m=0$ on $SU(2)$ with left-invariant metric. 

%
\begin{prop}
Consider $SU(2)$. If $g$ is a left-invariant metric and if $X$ is a left-invariant vector field, then there exist solutions to $\ric_X^m=Ag$ if and only if either $m>0$ with $A$ any real number or $m<0$ with $A>0$.
\end{prop}

\begin{proof}
Let $X=a_1X_1+a_2X_2+a_3X_3$. By Lemma \ref{lemma: aiaj=0}, at least two $a_i$'s must be zero. By Corollary \ref{cor:Unimodular leftinvariant killingfield}, $X$ is a Killing field, so we compute $\mathcal{L}_Xg$ using Proposition \ref{prop:lie derivative computation} as follows:

\begin{equation}\label{eq:lie derivative}
\setlength{\jot}{0.5em}
\begin{aligned}
&\mathcal{L}_Xg(X_1,X_2)= a_3(\lambda_1-\lambda_2) \\
&\mathcal{L}_Xg(X_2,X_3)= a_1(\lambda_2-\lambda_3)\\
&\mathcal{L}_Xg(X_1,X_3)=a_2(\lambda_3-\lambda_1).
\end{aligned}
\end{equation}
\vspace{1em}

By Table \ref{table: ricci signature},
the Ricci form is either $(+,+,+)$, $(+,0,0)$, or $(+,-,-)$. Let $|\ric(X_i,X_i))=\rho_i$ for $i=1,2,3$. If the Ricci form is $(+,+,+)$, then we have the following computations for $\ric_X^m$:

\begin{align*}
&\ric_X^m(X_1,X_1)=\rho_1-\frac{1}{m}a_1^2\\[0.5ex]
&\ric_X^m(X_2,X_2)=\rho_2-\frac{1}{m}a_2^2\\[0.5ex]
&\ric_X^m(X_3,X_3)=\rho_3-\frac{1}{m}a_3^2
\end{align*}
\vspace{1em}

%
%

Setting $\ric_X^m=Ag$, if all three $a_i$'s are zero, then $X=0$ and $\ric_X^m=\rho g$ where $\rho=\rho_1=\rho_2=\rho_3$.
\vspace{1em}

If $a_1=a_2=0$ and $a_3\neq 0$, and $\rho=\rho_1=\rho_2$, then $$X=\pm\sqrt{m(\rho_3-\rho)}X_3.$$ 
\vspace{1em}

Similarly, if $a_1=a_3=0$, and $\rho=\rho_1=\rho_3$, then $$X=\pm\sqrt{m(\rho_2-\rho)}X_2.$$
\vspace{0.5em}

If $a_2=a_3=0$, and $\rho=\rho_2=\rho_3$, then $$X=\pm\sqrt{m(\rho_1-\rho)}X_1.$$ In these cases, $\ric_X^m=\rho g$, where $\rho>0$, and $m$ can be positive or negative, depending on the sign of $\rho_3-\rho$, $\rho_2-\rho$, and $\rho_1-\rho$, respectively.
\vspace{1em}

If the Ricci form is $(+,0,0)$, then:

\begin{align*}
&\ric_X^m(X_1,X_1)=\rho_1-\frac{1}{m}a_1^2\\[0.5ex]
&\ric_X^m(X_2,X_2)=-\frac{1}{m}a_2^2\\[0.5ex]
&\ric_X^m(X_3,X_3)=-\frac{1}{m}a_3^2
\end{align*}
\vspace{1em}
%
%

The solutions to the above equations are $X=\pm\sqrt{\rho m}X_1$ and $\ric_X^m=0$. In this case, $m$ must be positive. 
\vspace{1em}

If the Ricci form is $(+,-,-)$, then 

\begin{align*}
&\ric_X^m(X_1,X_1)=\rho_1-\frac{1}{m}a_1^2\\[0.5ex]  
&\ric_X^m(X_2,X_2)=-\rho_2-\frac{1}{m}a_2^2\\[0.5ex]
&\ric_X^m(X_3,X_3)=-\rho_3-\frac{1}{m}a_3^2
\end{align*}
\vspace{1em}
%
%

Setting $\ric_X^m=Ag$, the solutions are $X=\pm\sqrt{m(\rho+\rho_1)}X_1$, where $\rho=\rho_2=\rho_3$. In this case, $\ric_X^m=-\rho g$ and $m$ must be positive.

\end{proof}

\section[]{Relation to Splitting Theorem, Myers' Theorem and Bochner's Theorem}\label{section: Splitting theorem and myers theorem}

According to Khuri-Woolgar-Wylie, the Splitting Theorem holds for $\ric_X^m$ if $m>0$ \cite[Theorem~2]{KhWoWy}. We also recall that if $(M,g)$ is a noncompact homogenous space, then it contains a line. Using the $\ric_X^m$ version of the Splitting Theorem and the fact about noncompact homogeneous spaces, we will show that of the 9 geometries which are 3-dimensional and homogeneous, the ones which don't split don't have solutions if $m>0$ and $A\geq 0$. 
\vspace{1em}

\begin{prop}$H^3$, $\widetilde{SL_2\mathbb{R}}$, $Nil$,$E(2)$, $H^2\times \mathbb{R}$, and $E(1,1)$ do not admit metrics such that $\ric_X^m=Ag$ for $m>0$ and $A\geq 0$.
\end{prop}

\begin{proof}
$H^3$, $\widetilde{SL_2\mathbb{R}}$, $Nil$,$E(2)$, and $E(1,1)$ all admit lines and don't split as $N\times\mathbb{R}$. Thus, the proposition follows by the Bakry \'Emery Ricci version of the Splitting Theorem by Khuri-Woolgar-Wylie.
\vspace{1em}

In the case of $H^2\times\mathbb{R}$, by the Splitting Theorem, $\ric_X^m\geq 0$ with $m>0$ if and only if $\ric_X^m\geq 0$ with $m>0$ on $H^2$. $H^2$ admits lines and doesn't split as $N\times \mathbb{R}$, so the proposition follows.
\end{proof}
\vspace{1em}

In \cite[Theorem~5]{ZhongminQian}, Qian proves that Myers' Theorem holds for gradient $m$-Bakry-\'Emery Ricci curvature when $m>0$. Limoncu showed in \cite[Theorem~1.2]{Limoncu} that Myers' Theorem holds for non-gradient $m$-Bakry-\'Emery Ricci curvature when $m>0$. In \cite{PhysicsMyersTheorem} Khuri-Woolgar use Limoncu's version of Myers' Theorem to study Near Horizon Geometries. Using this version of Myers' Theorem, we see that since $S^2\times \mathbb{R}$ and $\mathbb{R}^3$ are both noncompact, $S^2\times \mathbb{R}$ and $\mathbb{R}^3$ do not admit metrics such that $\ric_X^m=Ag$ for $m>0$ and $A>0$. In fact, since $SU(2)$ is the only compact simply-connected three-dimensional geometry, it is the only one that can admit a metric such that $\ric_X^m=Ag$ for $m>0$ and $A>0$.
\vspace{1em}

Next, we will discuss the $m<0$, $A<0$ case of the $m$-quasi Einstein metric. Bochner proved that if $(M,g)$ is compact, oriented and if $\ric<0$, then there are no nontrivial Killing fields (See \cite[Theorem~36]{Petersen}). This leads us to the next proposition.

\begin{prop}
If $M^n$ is a compact locally homogeneous Riemannian, and if $M^n$ is a compact quotient of a Lie group, $G$, then there are no solutions to $\ric_X^m=Ag$ if $m<0$ and $A<0$. 
\end{prop}

\begin{proof}
By Lemma \ref{theorem: ric_X^m and chen without ricci}, $\widetilde{X}$ is Killing on $G$. Then, $\ric=A\widetilde{g}+\frac{1}{m}\widetilde{X}^*\otimes \widetilde{X}^*$ which is negative, giving us a contradiction by Bochner's Theorem.
\end{proof}

\begin{corollary}
If $M^3$ is a compact locally homogeneous Riemannian manifold which satisfies $\ric_X^m=Ag$ with $m<0$ and $A<0$, then $M^3$ cannot be a compact quotient of $\mathbb{R}^3$, $SU(2)$, $\widetilde{SL_2(\mathbb{R})}$, $Nil$, $E(1,1)$, $H^2\times \mathbb{R}$, or $E(2)$.
\end{corollary}

\section[]{$m$-Quasi Einstein Equation on Geodesics}\label{section: analyze new equation}

Our next definition and proposition deal with analyzing the equation $\displaystyle\frac{1}{2}\mathcal{L}_Xg-\frac{1}{m}X^*\otimes X^*=Ag$, which we will use to find $m$-quasi Einstein solutions to $S^2\times \mathbb{R}$ and $H^3$. We will also prove theorems for more general spaces using this analysis.

\begin{defn} Let $\gamma(t)$ be a unit speed geodesic. We define $\varphi_\gamma(t)$ as $g(X_{\gamma(t)},\dot\gamma(t))$. Note that $\varphi_\gamma(t)$ is well defined for all $t$ that $\gamma(t)$ is defined. If it is clear which $\gamma(t)$ we are defining $\varphi_\gamma(t)$ along, then we will call our function $\varphi(t)$ rather than $\varphi_\gamma(t)$.
\end{defn}

\begin{prop}\label{prop:lambda m >0 =0 <0}
Let $(M,g)$ be a complete Riemannian manifold and let $\gamma: (-\infty,\infty)\rightarrow M$ be a unit speed geodesic. Suppose the equation $$\frac{1}{2}\mathcal{L}_Xg(\dot\gamma,\dot\gamma)-\frac{1}{m}g(X,\dot\gamma)g(X,\dot\gamma)=\lambda g(\dot\gamma,\dot\gamma)$$ is satisfied at every point on $\gamma$.
\begin{enumerate}\setlength\itemsep{1em}
    \item If $\lambda=0$ for $m\neq 0$ at every point along $\gamma$, then $\varphi(t)=0$.

    \item If $\lambda m>0$ at every point along $\gamma$, then there are no complete solutions to $\frac{1}{2}\mathcal{L}_X g-\frac{1}{m}X^*\otimes X^*=\lambda g$.
    \item If $\lambda m<0$ along a geodesic, then 
    \begin{center}
        $\varphi(t)=\sqrt{-\lambda m}\tanh\bigg(\frac{\sqrt{-\lambda m}}{m}(t+C)\bigg)$ or \\
        $\varphi(t)=\pm\sqrt{-\lambda m}$.
    \end{center}
    
\end{enumerate}
\end{prop}

\begin{proof}

We have the following set of equations:

\begin{equation*}
\begin{aligned}
\frac{d}{dt}(\varphi(t))&=\frac{1}{2}\mathcal{L}_X g(\dot\gamma,\dot\gamma)\\
&=\frac{1}{m}(X^*\otimes X^*)(\dot\gamma,\dot\gamma))+\lambda g(\dot\gamma, \dot\gamma)\\
&=\frac{1}{m}g(X,\dot\gamma)^2+\lambda\\
&=\frac{1}{m}\varphi^2(t)+\lambda.
\end{aligned}
\end{equation*}
\vspace{1em}

The proposition follows from Lemma \ref{lemma:function}.
\end{proof}
\vspace{1em}

\begin{remark}
If $M^n$ is a compact manifold, then we can prove Proposition \ref{prop:lambda m >0 =0 <0}(2) by using the Divergence Theorem. Taking the trace of both sides of $\displaystyle\frac{1}{2}\mathcal{L}_Xg-\frac{1}{m}X^*\otimes X^*=\lambda g$, we get $div(X)-\displaystyle\frac{1}{m}|X|^2=\lambda n$. Integrating both sides over $M$, we get

\begin{equation*}
\begin{aligned}
\displaystyle\int_M |X|^2&=-\int_M \lambda m n\\
&=-\lambda m n \vol(M)\\
\end{aligned}
\end{equation*}

Either $X=0$ and $\lambda=0$ or the left hand side is positive which implies $\lambda m$ must be negative. 
\end{remark}

In the following example, we provide an example of a manifold which satisfies $\ric_X^m=\lambda g$ with $\lambda m<0$.

\begin{ex} Let $M=S^1$ with the usual metric with $\{\frac{\partial}{\partial\theta}\}$ the basis vector. Let $X=\sqrt{-\lambda m}\frac{\partial}{\partial\theta}$ with $\lambda m<0$. Since $X$ is Killing and $S^1$ is Ricci flat, we get $\ric_X^m=\lambda g$. 

\end{ex}

Next, we give a global analysis of $\frac{1}{2}\mathcal{L}_Xg-\frac{1}{m}X^*\otimes X^*=\lambda g$ when $\lambda m<0$. In order to do this, we will first state a definition of critical point originally defined by Grove-Shiohama (Also see \cite{Petersen}). 
\vspace{1em}

\begin{defn}\cite{Petersen}\label{defn: criticalpoint}
Fix $p\in M$. A point $q$ is a critical point of the distance function to $p$ (is critical point to $p$) if, for every vector $V\in T_q M$, there is a minimal geodesic $\gamma$ with $\gamma(0)=p$, $\gamma(d(p,q))=q$ such that $g(\dot\gamma(d(p,q)),V)\leq 0$. 
\end{defn}
\vspace{1em}

\begin{lemma}\cite[Corollary~43]{Petersen}\label{lemma: criticalpointRn}
Suppose that there are no critical points of the distance function to $p$ in the annulus $\{q:a\leq d(p,q)\leq b\}$. Then $B(p,a)$ is homeomorphic to $B(p,b)$ and $B(p,b)$ deformation retracts onto $B(p,a)$. Moreover, if there are no critical points of $p$ in $M$, then $M$ is diffeomorphic to $\mathbb{R}^n$.
\end{lemma}
\vspace{1em}

Using similar techniques to those of Wylie in the proof of \cite[Proposition~1]{Wylie}, we will look for spaces which admit $\displaystyle\frac{1}{2}\mathcal{L}_Xg-\frac{1}{m}X^*\otimes X^*=\lambda g$ with $\lambda m<0$ everywhere. We will find that the only possibility is $S^1$ if the space is compact.
\vspace{1em}

\begin{prop}\label{prop: lambda_m<0 global condition}
If $M$ is a compact manifold which satisfies $\displaystyle\frac{1}{2}\mathcal{L}_Xg-\frac{1}{m}X^*\otimes X^*=\lambda g$ with $X\neq 0$ and $\lambda m<0$  along every geodesic, then $M=S^1$. 
\end{prop}

\begin{proof}
Since $M$ is compact, the function $f(p)=|X(p)|^2$ achieves a maximum and a minimum value. At the minimum, $0=D_X f=D_Xg(X,X)=2\mathcal{L}_X g(X,X)$. Then,

\begin{equation*}
\begin{aligned}
\frac{1}{2}\mathcal{L}_Xg(X,X)-\frac{1}{m}(X^*\otimes X^*)(X,X)&=\lambda g(X,X)\\
\Rightarrow -\frac{1}{m}|X|^4&=\lambda |X|^2.
\end{aligned}
\end{equation*}
\vspace{1em}

Then, either $|X|^2=-\lambda m$ or $|X|^2=0$ at the minimum point. By a similar argument, $|X|^2=-\lambda m$ or $|X|^2=0$ at the maximum point as well. Thus, either $|X|^2=-\lambda m$ for every point on $M$, or there exists a point $p\in M$ where $X(p)=0$.
\vspace{1em}

If $|X|^2=-\lambda m$ for every point in $M$, then taking the trace of $\displaystyle\frac{1}{2}\mathcal{L}_Xg-\frac{1}{m}X^*\otimes X^*=\lambda g$, we get $$\divergence(X)-\displaystyle\frac{|X|^2}{m}=\lambda n.$$ Plugging in $|X|^2=-\lambda m$, we get that $$\divergence(X)=\lambda (n-1).$$ Taking the integral of both sides over $M$ and using the Divergence Theorem, we get that $\lambda (n-1)\vol(M)=0$. If $\lambda=0$ then $X=0$ by Proposition \ref{prop:lambda m >0 =0 <0}(1), so  $n$ must be $1$. Since $M$ is compact, this means that $M=S^1$.
\vspace{1em}

In the case when there exists a point $p\in M$ such that $X(p)=0$, we will prove that there are no critical points to $p$ in $M$ and we will use Lemma \ref{lemma: criticalpointRn} to show that $M$ must be $\mathbb{R}^n$.
\vspace{1em}

By Definition \ref{defn: criticalpoint}, we want to show that there exists a vector $V$ such that every geodesic $\gamma$ with $\gamma(0)=p$, $\gamma(d(p,q))=q$ such that $g(\dot\gamma(d(p,q),V)>0$.
Consider the case when $m<0$. Let $\gamma(t)$ be a geodesic with $\gamma(0)=p$ and let $V=X$. If $\varphi(t)=g(X_{\gamma(t)},\dot\gamma(t))$, then since $X(p)=0$, $\varphi(0)$ must be $0$, so $\varphi(t)$ cannot be constantly nonzero.
\vspace{1em}

Then by Proposition \ref{prop:lambda m >0 =0 <0},  $$\varphi(t)=\sqrt{-\lambda m}\tanh\bigg(\displaystyle\frac{\sqrt{-\lambda m}}{m}t\bigg).$$

If $\varphi(t)=\sqrt{-\lambda m}\tanh\bigg(\displaystyle\frac{\sqrt{-\lambda m}}{m}t\bigg)$, then $\varphi(t)>0$ when $t>0$, so by Lemma \ref{lemma: criticalpointRn}, $M=\mathbb{R}^n$. This is a contradiction because $M$ is compact. 
\vspace{1em}

If $m>0$, then we again let $\gamma(t)$ be a geodesic with $\gamma(0)=p$. We will let $V=-X$ so that the differential equation we have to solve is $-\displaystyle\frac{d}{dt}\varphi(t)=\frac{1}{m}\varphi^2(t)+\lambda$. Then we get that the solutions are 

$$\varphi(t)=\sqrt{-\lambda m}\tanh
\bigg(\frac{-\sqrt{-\lambda m}}{m}t\bigg)\text{ or }\varphi(t)=\pm\sqrt{-\lambda m}.$$

$\varphi(t)$ cannot be $\pm\sqrt{-\lambda m}$ as in the $m<0$ case. If $\varphi(t)=\displaystyle\sqrt{-\lambda m}\tanh
\bigg(\frac{-\sqrt{-\lambda m}}{m}t\bigg)$, then $\varphi(t)$ is positive for $t>0$, giving us a contradiction by Lemma \ref{lemma: criticalpointRn}. 
\end{proof}

\begin{prop}On $H^3$, $\ric=-\rho g$ where $\rho>0$. $\ric_X^m=Ag$ if and only if $A+\rho=0$ and $X=0$.
\end{prop}

\begin{proof}
If $(A+\rho)m>0$, then by Proposition \ref{prop:lambda m >0 =0 <0}, there are no solutions. If $(A+\rho)m<0$, then by Proposition \ref{prop: lambda_m<0 global condition}, there are no solutions. If $A+\rho=0$, then by Proposition \ref{prop:lambda m >0 =0 <0}, $X=0$.
\end{proof}

\begin{corollary}
There are no solutions to $\ric_X^m=Ag$ with $A>0$ on a compact hyperbolic manifold.
\end{corollary}

Next, we give an example of a space $(M,g)$ which is non Euclidean,  $m$-quasi Einstein and Einstein, and $X$ is not trivial.

\begin{ex}
Consider $H^2$ with the metric $g=dr^2+e^{2r}dx^2$ and let $X=-m\frac{\partial}{\partial r}$. Then we have the following:\\

$\nabla_{\frac{\partial}{\partial r}}\frac{\partial}{\partial x}=\frac{\partial}{\partial x}$\\

$\nabla_\frac{\partial}{\partial x}\frac{\partial}{\partial x}=-e^{2r}\frac{\partial}{\partial r}$\\

$\nabla_{\frac{\partial}{\partial r}}\frac{\partial}{\partial r}=0$.\\

Then, we have the following computations for the Ricci curvature:\\

$\ric(\frac{\partial}{\partial r},\frac{\partial}{\partial x})=0$\\

$\ric(\frac{\partial}{\partial r},\frac{\partial}{\partial r})=-1$\\

$\ric(\frac{\partial}{\partial x},\frac{\partial}{\partial x})=-e^{2r}$,\\

so we see that our metric satisfies $\ric=-1g$. We have the following computations for $\ric_X^m$:\\

$\ric_X^m(\frac{\partial}{\partial r},\frac{\partial}{\partial x})=0$\\

$\ric_X^m(\frac{\partial}{\partial r},\frac{\partial}{\partial r})=-1-\frac{1}{m}(-m)^2=-1-m$\\

$\ric_X^m(\frac{\partial}{\partial x},\frac{\partial}{\partial x})=e^{2r}(-1-m)$,\\

so we see that $\ric_X^m=(-1-m)g$.

\end{ex}
\vspace{1em}

We are now ready to solve for the solutions of the $m$-quasi Einstein equation for $S^j\times \mathbb{R}$ when $j\geq 2$.
\vspace{1em}

\begin{prop}\label{prop: sjtimesr}
Consider $S^j\times \mathbb{R}$ with the product metric and $j\geq 2$, $S^j$ with a constant curvature metric of Ricci curvature $\rho$, and $\mathbb{R}$ with the flat metric. Then there exists a nontrivial $m$-quasi Einstein metric, $\ric_X^m=Ag$ if and only if $A=\rho$ and $m<0$.
\end{prop}

\begin{proof}
Let $\{X_1, X_2, \frac{\partial}{\partial r}\}$ be an orthonormal basis where $\{X_1, X_2\}$ is in $TS^2$ and $\{\frac{\partial}{\partial r}\}$ is in $T\mathbb{R}$. 
\vspace{1em}

First, consider the case $A-\rho=0$. Let $\gamma_{S^2}$ be a great circle on $S^2$ since the geodesics on $S^2$ are the great circles. 
We apply Proposition \ref{prop:lambda m >0 =0 <0} (1). This says that $X$ restricted to $S^2$ must be $0$. Letting $\gamma_{\mathbb{R}}$ be a unit speed geodesic in $\mathbb{R}$, we have $$\frac{1}{2}\mathcal{L}_Xg(\dot\gamma_{\mathbb{R}},\dot\gamma_{\mathbb{R}})-\frac{1}{m}X^*\otimes X^*(\dot\gamma_{\mathbb{R}},\dot\gamma_{\mathbb{R}})=A=\rho.$$ If $A-\rho=0$ and $m<0$, then by Proposition \ref{prop:lambda m >0 =0 <0}(3), $\varphi_{\gamma_{\mathbb{R}}}(t)$ is either $$\sqrt{-\rho m}\text{ or }\sqrt{-\rho m}\tanh\big(\frac{\sqrt{-\rho m}}{m}(t+C)\big)$$  which implies $$X=\sqrt{-\rho m}\frac{\partial}{\partial r}\text{ or }\sqrt{-\rho m}\tanh\big(\frac{\sqrt{-\rho m}}{m}(t+C)\big)\frac{\partial}{\partial r}.$$ If $A-\rho=0$ and $m>0$, then by Proposition \ref{prop:lambda m >0 =0 <0}(2), there are no solutions.
\vspace{1em}

If $(A-\rho)m>0$, then applying Proposition \ref{prop:lambda m >0 =0 <0}(2) to $\gamma_{S^2}$ in a similar fashion, we get that there are no solutions. 
\vspace{1em}

Consider the case $(A-\rho)m<0$. Since $S^2$ has dimension greater than 1, we can choose $\gamma_{S^2}$ perpendicular to $X$ at $0$ so that $\varphi_{\gamma_{S^2}}(0)=0$. and we apply Proposition \ref{prop:lambda m >0 =0 <0}(3) to $\gamma_{S^2}\in S^2$. Then $\varphi_{S^2}(t)$ is either $$\pm\sqrt{-(A-\rho) m}\text{ or }\sqrt{-(A-\rho) m}\tanh\bigg(\frac{\sqrt{(A-\rho) m}}{m}(t+C)\bigg).$$   $\varphi_{S^2}(t)$ cannot be $\displaystyle\sqrt{-(A-\rho) m}\tanh\bigg(\frac{\sqrt{(A-\rho) m}}{m}(t+C)\bigg)$ since $\gamma_{S^2}$ must be periodic and $\varphi_{S^2}(t)$ cannot be $\sqrt{-(A-\rho) m}$ since $\varphi_{\gamma_{S^2}}(0)=0$. This is a contradiction, so there are no solutions in this case as well.
\vspace{1em}
\end{proof}

Now, we will generalize Proposition \ref{prop: sjtimesr} to compact quotients of manifolds of the form $M\times N$, where $M$ and $N$ are Einstein manifolds. We prove this in a different way from Proposition \ref{prop: sjtimesr} because we cannot use the argument that $\varphi(t)$ must be periodic on $S^j$.
\vspace{1em}

\begin{lemma}\label{lemma:either positive or one-dimensional}
Consider a compact quotient of $M\times N$ with the product metric where $M$ is an Einstein manifold. If there is a nontrivial $m$-quasi Einstein solution on such a space, then either $X|_M=0$ or $M$ is one-dimensional. 
\end{lemma}

\begin{proof}
 Without loss of generality, assume that $M$ and $N$ are simply connected because if either space is not simply connected, we can lift them to the universal cover. Let $\pi: M\times N\rightarrow \bigslant{(M\times N)}{\Gamma}$ be the universal covering map and let $\ric_M=\rho_M g_M$. Let $\gamma_M(t)$ be a unit speed geodesic in $M$. Then we have $$\frac{1}{2}\mathcal{L}_Xg(\dot\gamma_M,\dot\gamma_M)-\frac{1}{m}X^*\otimes X^*(\dot\gamma_M,\dot\gamma_M)=A-\rho_M.$$

We aim to show that either $A-\rho_M=0$ or $M=\mathbb{R}$. If $M$ is not $\mathbb{R}$ then $M$ is not one-dimensional, so we can choose $\gamma_M$ to be perpendicular to $X$ at $0$. In this case, $\varphi_{\gamma_M}(0)$ is zero, so $\varphi_{\gamma_M}(t)$ cannot be constantly nonzero. If $(A-\rho_M)m>0$, then by Proposition \ref{prop:lambda m >0 =0 <0}(2), there are no complete solutions. If $(A-\rho_M)m<0$, then by Proposition \ref{prop:lambda m >0 =0 <0}(3), and since $\varphi_{\gamma_M}(t)$ $\varphi_{\gamma_M}(t)$ is $$\sqrt{-(A-\rho_M) m}\tanh\bigg(\frac{\sqrt{(A-\rho_M) m}}{m}(t+C)\bigg).$$ 
To show that $\varphi_{\gamma_M}(t)$ cannot be $\displaystyle\sqrt{-(A-\rho_M) m}\tanh\bigg(\frac{\sqrt{(A-\rho_M) m}}{m}(t+C)\bigg)$, we will use an argument similar to the proof of Lemma \ref{lemma:chen_without_ricci}.
\vspace{1em}

Consider the set $\overline{\{\pi\circ \gamma_M(t):t\in\mathbb{R}\}}$. Since this set is closed, $\varphi_{\gamma_M}(t)$ has a maximal point, $t_{max}$ on this set. Because the supremum of the $\tanh$ function is $1$, we know that the maximum of $\varphi_{\gamma_M}(t)$ on $\overline{\{\pi\circ \gamma_M(t):t\in\mathbb{R}\}}$ is $\sqrt{-(A-\rho_M) m}$.
\vspace{1em}

Let $\beta(t)$ be a geodesic of $X$ such that $\beta(0)=\gamma_M(t_{max})=\sqrt{-(A-\rho_M) m}$. Now consider the set $\{\pi\circ \beta(t):t\in\mathbb{R}\}$. Along $\beta(t)$, $\varphi_{\beta}(t)$ is either $\sqrt{-(A-\rho_M) m}$ or $-\sqrt{-(A-\rho_M) m}\tanh(\frac{\sqrt{-(A-\rho_M) m}}{m}(t+C))$. Since the supremum of $\varphi_{\beta}(t)$ on $\{\beta(t):t\in\mathbb{R}\}$ is $\sqrt{-(A-\rho_M) m}$ and the $\tanh$ function never achieves its maximum on its domain, $\varphi_{\beta}(t)$ must be constantly $\sqrt{-(A-\rho_M) m}$ on the set $\{\pi\circ \beta(t):t\in \mathbb{R}\}$.
\vspace{1em}

Finally, since $\overline{\{\pi\circ \beta(t):t\in\mathbb{R}\}}=\overline{\{\pi\circ \gamma_M(t):t\in\mathbb{R}\}}$, $\varphi_{\gamma_M}(t)$ is constant on $\overline{\{\pi\circ \gamma_M(t):t\in\mathbb{R}\}}$. Thus, $\varphi_{\gamma_M}(t)$ is constant.
\vspace{1em}

Since $\varphi_{\gamma_M}(0)=0$, $\varphi_{\gamma_M}(t)$ cannot be $\pm\sqrt{-(A-\rho_M)m}$, and so we have arrived at a contradiction.
\vspace{1em}

Thus, either $M=\mathbb{R}$ or $A-\rho_M=0$. If $A-\rho_M=0$, then by Proposition \ref{prop:lambda m >0 =0 <0}(1), $\varphi_{\gamma_M}=0$, which implies $X|_M=0$.
\end{proof}

Now we can prove Theorem \ref{theorem:compact einstein cross einstein}
\begin{proof}[Proof of Theorem \ref{theorem:compact einstein cross einstein}]

Let $\pi: M\times N\rightarrow \bigslant{(M\times N)}{\Gamma}$ be the universal covering map and let $\ric_M=\rho_M g_M$ and $\ric_N=\rho_N g_N$. Let $\gamma_M(t)$ be a unit speed geodesic in $M$ and let $\gamma_N(t)$ be a unit speed geodesic in $N$. By Lemma \ref{lemma:either positive or one-dimensional}, $M$ is either one-dimensional or $X|_M=0$ and $A-\rho_M=0$. By symmetry, either $A-\rho_N=0$ and $X|_N$ is zero, or $N=\mathbb{R}$. 
\vspace{1em}

Suppose without loss of generality that $N=\mathbb{R}$. Then $$\frac{1}{2}\mathcal{L}_Xg(\dot\gamma_N,\dot\gamma_N)-\frac{1}{m}X^*(\dot\gamma_N)X^*(\dot\gamma_N)=Ag.$$ 

By Proposition \ref{prop:lambda m >0 =0 <0}, $A=0$, then $$X=0,$$ If $Am>0$, then there are no solutions, and if $Am<0$, then  $$X=\sqrt{-\lambda m}\tanh\bigg(\frac{\sqrt{-\lambda m}}{m}(t+C)\bigg)\frac{\partial}{\partial r}\text{ or }X=\pm\sqrt{-\lambda m}\frac{\partial}{\partial r}.$$ 

If we consider the set $\overline{\{\pi\circ\gamma_N(t):t\in\mathbb{R}\}}$ and use the same argument as above, we see that $X=\sqrt{-\lambda m}\tanh\bigg(\frac{\sqrt{-\lambda m}}{m}(t+C)\bigg)\frac{\partial}{\partial r}$ is not a solution. 
\vspace{1em}

Thus, the only solutions are $X=0$ when $A=\rho_M=\rho_N\neq 0$, and $X=\pm\sqrt{-Am}\frac{\partial}{\partial r}$ when either $N=\mathbb{R}$ or $M=\mathbb{R}$.

\end{proof}

\section{Summary}\label{section:table}
In the following table, we summarize the solutions of locally homogeneous compact three-manifolds, $M^3$ which have quasi-Einstein metrics. In the first column, which we've named ``Manifold", we have the manifolds which act cocompactly on $M^3$. The second through seventh columns are the different signs of $m$ and $A$ in our $m$-quasi Einstein equation, $\ric_X^m=Ag$. If there are no solutions to the compact quotient of ``Manifold", we write \textit{None}. If the only solutions are when $X=0$, then we say \textit{Trivial solution}, and if there are nontrivial solutions, then we say \textit{Exists}.
\vspace{1em}

\bgroup
\def\arraystretch{2.5}
\begin{center}
\begin{tabular}{ c | c | c | c | c | c | c}
Manifold & \parbox{1.5cm}{$m>0$\\ $A>0$} &
\parbox{1.5cm}{$m>0$\\ $A=0$} &
\parbox{1.5cm}{$m>0$\\ $A<0$} &  \parbox{1.5cm}{$m<0$\\ $A>0$} &
\parbox{1.5cm}{$m<0$\\ $A=0$} & 
\parbox{1.5cm}{$m<0$\\ $A<0$}\\
\hline
$\mathbb{R}^3$ & None & \parbox{1.5cm}{Trivial Solution} & None & None & \parbox{1.5cm}{Trivial Solution} & None\\

$SU(2)$ & Exists & Exists & Exists & Exists & None & None\\

$\widetilde{SL_2(\mathbb{R})}$ & None & None & None & None & Exists & None \\

$Nil$ & None & None & Exists  & None & None & None\\

$E(1,1)$ & None & None & None & None & None & None\\

$E(2)$ & None & None & None & None & None & None\\

$H^2\times \mathbb{R}$ & None & None & Exists & None & None & None\\

$S^2\times\mathbb{R}$ & None & None & None &  Exists & None & None\\

$H^3$ & None & None & \parbox{1.5cm}{Trivial Solution} & None & None & \parbox{1.5cm}{Trivial Solution}\\

\end{tabular}
\end{center}

\section*{Acknowledgements}
The author would like to thank her thesis advisor, Professor William Wylie, for all of his help and support in writing this paper.\\

This work was partially supported by NSF grant DMS-1654034. 
\bibliographystyle{plain}
\bibliography{su(2)}

\end{document}